\documentclass[11pt]{article}
\usepackage[total={6in, 9in}]{geometry}
 \usepackage{amsfonts}
\usepackage{amsthm}
\usepackage{amssymb}
\usepackage{amsmath}
\usepackage{mathrsfs}
\usepackage{cases}
\usepackage{latexsym,bm}
\usepackage{indentfirst}
\usepackage{color}
\usepackage{ifpdf}
\usepackage{graphicx}
\usepackage{psfrag}
\usepackage{thmtools}
\usepackage[
pdfauthor={},
pdftitle={},
pdfstartview=XYZ,
bookmarks=true,
colorlinks=true,
linkcolor=blue,
urlcolor=blue,
citecolor=blue,
bookmarks=true,
linktocpage=true,
hyperindex=true
]{hyperref}

\usepackage[natural]{xcolor}
\usepackage{xurl}
\usepackage{tikz}
\usepackage{subfigure}
\usepackage{tikz-3dplot}
\usepackage{pgf,pgfplots}
\pgfplotsset{compat=1.18}
\usepackage{enumerate} 

\newtheorem{theorem}{\textbf{Theorem}}[section]

\newtheorem{lemma}[theorem]{\textbf{Lemma}}
\newtheorem{claim}{\textbf{Claim}}
\newtheorem{conjecture}[theorem]{\textbf{Conjecture}}

\newcommand{\biburl}[1]{\url{#1}}

\usepackage{pdflscape}

\begin{document}
\title{Degree sequence condition for  pancyclicity in tough graphs}

\author{Songling Shan\footnote{Auburn University, Department of Mathematics and Statistics, Auburn, AL 36849.
		Email: {\tt szs0398@auburn.edu}.   
		Supported in part by NSF grant DMS-2451895.}
	\qquad 
	Zach Warren \footnote{Auburn University, Department of Mathematics and Statistics, Auburn, AL 36849.
		Email:	{\tt zzw0115@auburn.edu}. }
}

\date{September 18, 2026}
\maketitle

\begin{abstract}
   Let $t \ge 1$ be an integer, and let $G$ be a $t$-tough $n$-vertex graph with degree sequence $d_1, d_2, \ldots, d_n$ in non-decreasing order.   
   In 1995, Ho\`ang conjectured that if $G$ is Hamiltonian and, for every integer $i$ satisfying $t\le i<n/2$, $d_i\le i$, and $d_{n-i+t}<n-i$, one has $d_j + d_{n-j+t} \ge n$ for all $j$ with $i < j < \frac{n}{2}$, then $G$ is pancyclic or bipartite. 
   In this paper, we disprove the conjecture for $t = 1$ and confirm it for all $t \ge 7$.
\end{abstract}


\section{Introduction}

We consider only finite simple graphs. Let $G$ be a graph on $n$
vertices. Denote by $V(G)$ and $E(G)$ the vertex set and edge set
of $G$, respectively. For a vertex $v\in V(G)$, let $d_G(v)$
denote the degree of $v$ in $G$. We omit the subscript $G$ when
the graph is clear from the context. The degrees of the vertices
of $G$ arranged in non-decreasing order form the \emph{degree
	sequence} of $G$, denoted by $d_1,d_2,\ldots,d_n$. A graph is
\emph{Hamiltonian} if it contains a cycle through all its vertices.
In 1972, Chv\'atal~\cite[Theorem 1]{chvatal1972} proved the following degree
sequence condition for Hamiltonicity.

\begin{theorem}[{Chv\'atal~\cite[Theorem 1]{chvatal1972}}]
	\label{thm_chvatal-deg-seq}
	Let $G$ be a graph on $n\ge 3$ vertices with degree sequence
	$d_1,d_2,\ldots,d_n$. If, for every integer $i$ with $1\le i<n/2$,
	$d_i\le i$ implies $d_{n-i}\ge n-i$, then $G$ is Hamiltonian.
\end{theorem}

Let $c(G)$ denote the number of components of $G$. For a real
number $t\ge 0$, a graph $G$ is \emph{$t$-tough} if
$|S|\ge t\,c(G-S)$ for every $S\subseteq V(G)$ with $c(G-S)\ge 2$.
The \emph{toughness} of a noncomplete graph $G$, denoted by
$\tau(G)$, is the largest $t$ for which $G$ is $t$-tough.
If $G$ is complete, then $\tau(G)$ is defined to be $\infty$.
Chv\'atal~\cite{chvatal1973} introduced this concept in 1973.
Every Hamiltonian graph is $1$-tough. Using this necessary
condition, Ho\`ang~\cite{hoang1995} obtained several
generalizations of Theorem~\ref{thm_chvatal-deg-seq}, including
the following result.

\begin{theorem}[Ho\`ang~{\cite[Theorem~6]{hoang1995}}]
	\label{thm_hoang-main}
	Let $G$ be a $1$-tough graph on $n\ge 3$ vertices with degree
	sequence $d_1,d_2,\ldots,d_n$. Suppose that, for every integer
	$i$ with $1\le i<n/2$, $d_i\le i$ and $d_{n-i+1}<n-i$ imply
	$d_j+d_{n-j+1}\ge n$ for all integers $j$ with
	$i<j\le\lceil n/2\rceil$. Then $G$ is Hamiltonian.
\end{theorem}

A graph on $n$ vertices is \emph{pancyclic} if it contains a cycle
of every length from $3$ to $n$. Bondy's Meta
Conjecture~\cite{bondy1971,bondy1971meta} asserts that most
nontrivial sufficient conditions for Hamiltonicity also imply
pancyclicity, apart from a simple family of exceptional graphs.
In the setting of degree sequences, Schmeichel and
Hakimi~\cite[Theorem, p.~23]{schmeichel1974} proved that every graph satisfying
the hypothesis of Theorem~\ref{thm_chvatal-deg-seq} is pancyclic
or bipartite. Ho\`ang~\cite[Theorem~9]{hoang1995} showed that the same
conclusion holds under the hypotheses of
Theorem~\ref{thm_hoang-main}. More generally, he proved the
following theorem.

\begin{theorem}[Ho\`ang~{\cite[Theorem~9]{hoang1995}}]
	\label{thm_hoang-tough}
	Let $t\ge 1$ be an integer and $G$ be a Hamiltonian $t$-tough
	graph on $n\ge 3$ vertices with degree sequence
	$d_1,d_2,\ldots,d_n$. Suppose that, for every integer $i$ with
	$t\le i<n/2$, $d_i\le i$ and $d_{n-i+t}<n-i$ imply
	$d_j+d_{n-j+t}\ge n$ for all integers $j$ with
	$i<j\le\lceil n/2\rceil$. Then $G$ is pancyclic or bipartite.
\end{theorem}

Ho\`ang~\cite[Conjecture~5]{hoang1995} conjectured that
Theorem~\ref{thm_hoang-tough} remains true when the degree-sum
inequality is required only for $i<j<n/2$, rather than
$i<j\le\lceil n/2\rceil$.

The difference is the degree-sum inequality at
$j=\lceil n/2\rceil$. In particular, if
$i=\lceil n/2\rceil-1$, then there is no integer $j$ with
$i<j<n/2$, so the conjectured condition imposes no degree-sum
requirement at this index. Theorem~\ref{thm_hoang-tough}, however,
still requires $d_{\lceil n/2\rceil}+d_{\lfloor n/2\rfloor+t}\ge n$.
Our counterexamples in Section~\ref{sec_counterexample} exploit
the absence of this midpoint inequality.

\begin{conjecture}[{Ho\`ang~\cite[Conjecture~5]{hoang1995}}]
	\label{conj_hoang}
	Let $t\ge 1$ be an integer and $G$ be a Hamiltonian $t$-tough
	graph on $n\ge 3$ vertices with degree sequence
	$d_1,d_2,\ldots,d_n$. Suppose that, for every integer $i$ with
	$t\le i<n/2$, $d_i\le i$ and $d_{n-i+t}<n-i$ imply
	$d_j+d_{n-j+t}\ge n$ for all integers $j$ with $i<j<n/2$.
	Then $G$ is pancyclic or bipartite.
\end{conjecture}

In this paper, we disprove Conjecture~\ref{conj_hoang} for
$t=1$ and confirm it for all integers $t\ge 7$. Our first
result gives a counterexample of every order at least $10$.

\begin{theorem}
	\label{thm_counterexample}
	For every integer $n\ge 10$, there exists a Hamiltonian
	$1$-tough graph $G$ on $n$ vertices with degree sequence
	$d_1,d_2,\ldots,d_n$ such that, for every integer $i$ with
	$1\le i<n/2$, $d_i\le i$ and $d_{n-i+1}<n-i$ imply
	$d_j+d_{n-j+1}\ge n$ for all integers $j$ with $i<j<n/2$,
	but $G$ is neither pancyclic nor bipartite.
\end{theorem}

Shan and Tanyel~\cite[Theorem~4]{shan2026strengthening} proved that, for
every integer $t\ge 4$, a $t$-tough graph satisfying the degree
sequence condition in Conjecture~\ref{conj_hoang} is Hamiltonian.
Thus the Hamiltonicity assumption in the conjecture is redundant
for $t\ge 4$. Our second result confirms the conjecture for
$t\ge 7$ and can therefore be stated without assuming
Hamiltonicity. Since every bipartite graph on at least three
vertices has toughness at most $1$, the bipartite alternative
does not occur in this range.

\begin{theorem}
	\label{thm_main-theorem}
	Let $t\ge 7$ be an integer and $G$ be a $t$-tough graph on
	$n\ge 3$ vertices with degree sequence $d_1,d_2,\ldots,d_n$.
	Suppose that, for every integer $i$ with $t\le i<n/2$,
	$d_i\le i$ and $d_{n-i+t}<n-i$ imply $d_j+d_{n-j+t}\ge n$
	for all integers $j$ with $i<j<n/2$. Then $G$ is pancyclic.
\end{theorem}

\section{Preliminaries}
\label{sec_preliminaries}

We first state the Hamiltonicity result of Shan and Tanyel that
will be used in the proof of Theorem~\ref{thm_main-theorem}.

\begin{theorem}[Shan and Tanyel~{\cite[Theorem~4]{shan2026strengthening}}]
\label{thm_shan-tanyel-hamiltonicity}
Let $t\ge4$ be an integer and $G$ be a $t$-tough graph on
$n\ge3$ vertices with degree sequence $d_1,d_2,\ldots,d_n$.
Suppose that, for every integer $i$ with $t\le i<n/2$,
$d_i\le i$ and $d_{n-i+t}<n-i$ imply
$d_j+d_{n-j+t}\ge n$ for all integers $j$ with $i<j<n/2$.
Then $G$ is Hamiltonian.
\end{theorem}

Let $G$ be a graph. For $S\subseteq V(G)$, let $G[S]$ be the
subgraph of $G$ induced by $S$, and let $G-S=G[V(G)\setminus S]$.
We write $G-v$ for $G-\{v\}$. Let $N_G(v)$ denote the neighborhood
of $v$ in $G$, and let $N_G(S)=(\bigcup_{v\in S}N_G(v))\setminus S$.
As with degrees, we omit the subscript when the graph is clear from
the context. The independence number, vertex connectivity, and maximum
degree of $G$ are denoted by $\alpha(G)$, $\kappa(G)$, and $\Delta(G)$,
respectively. We use the conventions that $\kappa(G)=0$ when $G$ is
disconnected and $\kappa(K_n)=n-1$. For integers $p$ and $q$, let
$[p,q]=\{i\in\mathbb Z:p\le i\le q\}$.

We first record some consequences of toughness. The degree and
connectivity bounds will be used to obtain Hamilton paths with
prescribed endpoints, while the last inequality allows us to apply
a closure operation after deleting vertices.

\begin{lemma}
	\label{lem_toughness-observation}
	Let $t\ge 0$ be a real number and $G$ be a $t$-tough graph on $n$
	vertices. If $G$ is not complete, then
	$\delta(G)\ge\kappa(G)\ge 2t$ and $\alpha(G)\le n/(t+1)$.
	Moreover, $\tau(G-S)\ge t-|S|/2$ for every $S\subseteq V(G)$.
\end{lemma}

\begin{proof}
	Suppose first that $G$ is not complete. A minimum vertex cut leaves
	at least two components, so $\kappa(G)\ge 2t$. The inequality
	$\delta(G)\ge\kappa(G)$ is immediate. If $W$ is a maximum independent
	set, then $|W|\ge 2$ and $c(G-(V(G)\setminus W))=|W|$. Thus
	$n-|W|\ge t|W|$, giving $\alpha(G)\le n/(t+1)$.
	
	Now let $S\subseteq V(G)$. There is nothing to prove if $G-S$ is
	complete. Otherwise, for every $W\subseteq V(G)\setminus S$ with
	$c=c(G-(S\cup W))\ge 2$, toughness gives
	\[
	|W|\ge tc-|S|\ge \left(t-\frac{|S|}{2}\right)c.
	\]
	Taking the minimum over all such $W$ proves the last assertion.
\end{proof}

For a positive integer $s$, the \emph{$s$-closure} of $G$, denoted by
$C_s(G)$, is obtained by repeatedly adding an edge between two
nonadjacent vertices whose degree sum in the current graph is at
least $s$, until no such pair remains. In particular, if $H=C_s(G)$
and $xy\notin E(H)$, then $d_H(x)+d_H(y)<s$. We will use the following
closure lemma of Shan and Tanyel.

\begin{lemma}[{Shan and Tanyel~\cite[Theorem~5]{shan2025}}]
	\label{lem_shan-toughness-closure}
	Let $t\ge 4$ be a rational number and $G$ be a $t$-tough graph on
	$n\ge 3$ vertices. If $x$ and $y$ are nonadjacent vertices of $G$
	with $d_G(x)+d_G(y)\ge n-t$, then $G$ is Hamiltonian if and only if
	$G+xy$ is Hamiltonian.
\end{lemma}

Let $s\ge 0$ be an integer. A graph $G$ on at least $s+3$ vertices
is \emph{$s$-Hamiltonian} if $G-S$ is Hamiltonian for every
$S\subseteq V(G)$ with $|S|\le s$. The next lemma extends the
preceding closure result to $s$-Hamiltonicity. Its cases $s=1$ and
$s=3$ will both be used in the proof of Theorem~\ref{thm_main-theorem}.

\begin{lemma}
	\label{lem_s-hamiltonian-closure}
	Let $s\ge 0$ be an integer, let $t\ge 4+s/2$ be a rational number,
	and let $G$ be a $t$-tough graph on $n\ge s+3$ vertices. If
	$a$ is a positive integer with $a\ge n-t+3s/2$, then $G$ is $s$-Hamiltonian if and only if $C_a(G)$
	is $s$-Hamiltonian. In particular, if $C_a(G)$ is complete, then
	$G$ is $s$-Hamiltonian.
\end{lemma}

\begin{proof}
	Adding edges preserves $s$-Hamiltonicity, so we prove the converse.
	Let \[G=G_0,G_1,\ldots,G_m=C_a(G)\] be a sequence of graphs obtained
	in the closure operation. Fix $S\subseteq V(G)$ with $r=|S|\le s$.
	Since adding edges does not decrease toughness,
	Lemma~\ref{lem_toughness-observation} gives
	$\tau(G_i-S)\ge t-r/2\ge 4$ for every $i\in[0,m]$.
	
	Suppose that $G_{i+1}=G_i+xy$. If $x\in S$ or $y\in S$, then
	$G_{i+1}-S=G_i-S$. Otherwise,
	\[
	\begin{aligned}
		d_{G_i-S}(x)+d_{G_i-S}(y)
		&\ge a-2r\\
		&\ge n-t+\frac{3s}{2}-2r\\
		&\ge (n-r)-\left(t-\frac r2\right).
	\end{aligned}
	\]
	By Lemma~\ref{lem_shan-toughness-closure}, $G_i-S$ is Hamiltonian
	if and only if $G_{i+1}-S$ is Hamiltonian. Applying this at every
	step shows that $G-S$ is Hamiltonian whenever $C_a(G)-S$ is
	Hamiltonian. Since $S$ was arbitrary, the conclusion follows.
	The last assertion holds because a complete graph on at least
	$s+3$ vertices is $s$-Hamiltonian.
\end{proof}

For a $t$-tough graph $G$ on $n\ge 6$ vertices with $t\ge 7$,
Lemma~\ref{lem_s-hamiltonian-closure} has two consequences that we will use: if $C_{n-t+5}(G)$ is complete, then
$G$ is $3$-Hamiltonian; and if $C_{n-t+2}(G)$ is $1$-Hamiltonian,
then so is $G$.

We next recall results that turn Hamilton paths or Hamilton cycles
into cycles of all lengths. A graph is \emph{Hamiltonian-connected}
if every two distinct vertices are the endpoints of a Hamilton path.
The following theorem provides such paths through a comparison of
independence number and connectivity.

\begin{lemma}[{Chv\'atal and Erd\H{o}s~\cite[Theorem~3]{chvatal-erdos1972}}]
	\label{lem_chvatal-hamiltonian-connected}
	Let $G$ be a graph on at least three vertices. If
	$\alpha(G)\le\kappa(G)-1$, then $G$ is Hamiltonian-connected.
\end{lemma}

The next two results apply according to whether the endpoints of
a Hamilton path are adjacent.

\begin{lemma}[{Bondy~\cite{bondy1971}; Schmeichel and
	Hakimi~\cite[Lemma~1]{schmeichel-hakimi1988}}]
	\label{lem_bondy-hamilton-cycle}
	Let $G$ be a graph on $n\ge 3$ vertices with a Hamilton cycle
	$v_1v_2\cdots v_nv_1$. If $d_G(v_1)+d_G(v_n)>n$, then $G$ is
	pancyclic.
\end{lemma}

\begin{lemma}[{Faudree, Favaron, Flandrin, and Li~\cite[Theorem 2]{faudree1996}}]
	\label{lem_faudree-hamilton-path}
	Let $G$ be a graph on $n\ge 3$ vertices with a Hamilton path
	$v_1v_2\cdots v_n$. If $v_1v_n\notin E(G)$ and
	$d_G(v_1)+d_G(v_n)\ge n$, then $G$ is pancyclic.
\end{lemma}

We will also need a consequence of $s$-Hamiltonicity. Unlike the
preceding results, it requires a degree bound at only one vertex.

\begin{lemma}
	\label{lem_s-hamiltonian}
	Let $s\ge 1$ be an integer and $G$ be an $s$-Hamiltonian graph on
	$n\ge s+3$ vertices. If $\Delta(G)>(n-s)/2$, then $G$ is pancyclic.
\end{lemma}

\begin{proof}
	Let $x\in V(G)$ with $d_G(x)=\Delta(G)$, and choose a set $S$ of
	nonneighbors of $x$ of size
	$\min\{s-1,n-1-d_G(x)\}$. Put $H=G-S$ and $m=|V(H)|$.
	The graph $H-x$ has a Hamilton cycle, and $d_H(x)=d_G(x)$.
	If $|S|=s-1$, then $d_H(x)>(n-s)/2=(m-1)/2$.
	Otherwise, all nonneighbors of $x$ have been deleted, so
	$d_H(x)=m-1>(m-1)/2$.
	
	Let $C=u_0u_1\cdots u_{m-2}u_0$ be a Hamilton cycle of $H-x$,
	with subscripts taken modulo $m-1$. Fix $\ell\in[3,m]$.
	Since $|N_H(x)|>(m-1)/2$, there is an index $i$ for which both
	$u_i$ and $u_{i+\ell-2}$ belong to $N_H(x)$. Otherwise, the map
	$u_i\mapsto u_{i+\ell-2}$ would inject $N_H(x)$ into its complement
	in $V(C)$. The cycle
	$xu_iu_{i+1}\cdots u_{i+\ell-2}x$ has length $\ell$.
	Thus $H$ is pancyclic. Since $G$ is $s$-Hamiltonian, it also has
	cycles of every length from $n-s$ to $n$. As $m\ge n-s+1$, these
	two ranges contain every length from $3$ to $n$.
\end{proof}

The following three results of Ho\`ang provide pancyclicity from
a degree sequence condition or from the number of vertices of
large degree.

\begin{lemma}[{Ho\`ang~\cite[Theorem~7]{hoang1995}}]
	\label{lem_hoang-no-bad-index}
	Let $t\ge 1$ be an integer and $G$ be a Hamiltonian $t$-tough graph
	on $n\ge 3$ vertices with degree sequence $d_1,d_2,\ldots,d_n$.
	If, for every integer $i$ with $t\le i<n/2$, $d_i\le i$ implies
	$d_{n-i+t}\ge n-i$, then $G$ is pancyclic or bipartite.
\end{lemma}

For an integer $n\ge 4$ divisible by $4$, let $S_n$ be obtained
from a complete graph on $n/2$ vertices and a disjoint union of $n/4$ copies of $K_2$
by adding a perfect matching between their vertex sets.
Note that $\delta(S_n)=2$.

\begin{lemma}[{Ho\'ang~\cite[Theorem~8]{hoang1995}}]
	\label{lem_hoang-R-half}
	Let $G$ be a Hamiltonian graph on $n\ge 3$ vertices. If at least
	$n/2$ vertices of $G$ have degree at least $n/2$, then $G$ is
	pancyclic, bipartite, or isomorphic to $S_n$.
\end{lemma}

\begin{lemma}[{Ho\'ang~\cite[Lemma~10]{hoang1995}}]
	\label{lem_hoang-R-third}
	Let $G$ be a Hamiltonian graph on $n\ge 3$ vertices. If more than
	$n/3$ vertices of $G$ have degree greater than $n/2$, then $G$ is
	pancyclic.
\end{lemma}

Finally, we recall a generalized form of Hall's theorem. Akiyama
and Kano~\cite{akiyama2011} give an account of matching results of
this kind. We include the short reduction to Hall's theorem.  In our
application, $f$ is constantly $3$, and the resulting subgraph
provides pairwise disjoint sets of attachment vertices in a clique.

\begin{lemma}[Generalized marriage lemma]
	\label{lem_generalized-marriage}
	Let $G$ be a bipartite graph with bipartition $(X,Y)$, and let
	$f:X\to\mathbb Z_{\ge 0}$. Then $G$ has a spanning subgraph $F$
	such that $d_F(x)=f(x)$ for every $x\in X$ and $d_F(y)\le 1$
	for every $y\in Y$ if and only if
	$|N_G(S)|\ge\sum_{x\in S}f(x)$ for every $S\subseteq X$.
\end{lemma}

\begin{proof}
Necessity follows because the edges of $F$ incident with $S$
have distinct endpoints in $Y$.
For sufficiency, replace each $x\in X$ by $f(x)$ copies with
the same neighborhood in $Y$, and let $X'$ be the set of all
copies. For a set $U\subseteq X'$, let $S\subseteq X$ consist
of the vertices represented in $U$. The new bipartite graph
has $|N(U)|=|N_G(S)|\ge\sum_{x\in S}f(x)\ge|U|$.
Hall's theorem therefore gives a matching covering $X'$.
Identifying the copies of each $x$ gives the required edges of
$F$, and all remaining vertices are included as isolated vertices.
\end{proof}

\section{Proof of Theorem~\ref{thm_counterexample}}
\label{sec_counterexample}

\begin{proof}
	Let $n\ge 10$ be an integer and put $m=\lfloor n/2\rfloor$.
	We give a construction that works for both parities of $n$.
	Let $B=\{b_1,\ldots,b_m\}$, and partition $B$ into sets $B_1$
	and $B_2$ of sizes $\lfloor m/2\rfloor$ and $\lceil m/2\rceil$,
	respectively, so that $b_1,b_m\in B_1$ and $b_2,b_3\in B_2$.
	Such a partition exists because $m\ge 5$.
	Let $Z=\{z_1,\ldots,z_{n-m-2}\}$, and let $x$ and $y$ be two
	additional vertices.
	
	Define $G$ on $B\cup Z\cup\{x,y\}$ by making $Z$ complete to
	$B$, joining $x$ to every vertex of $B_1$, joining $y$ to every
	vertex of $B_2$, and adding the edge $xy$. These are all the
	edges of $G$. Equivalently, $G$ is obtained from the complete
	bipartite graph with parts $Z\cup\{x,y\}$ and $B$ by deleting
	the edges from $x$ to $B_2$ and from $y$ to $B_1$, and then
	adding $xy$.
	
	The graph $G-xy$ is bipartite, and $x$ and $y$ have no common
	neighbor. Thus $G$ is triangle-free and hence is not pancyclic.
	On the other hand, $xb_1z_1b_2yx$ is a cycle of length $5$, so
	$G$ is not bipartite.
	
	We next show that $G$ is Hamiltonian. If $n=2m$, then $|Z|=m-2$
	and a Hamilton cycle is
	\[
	xb_1z_1b_2yb_3z_2b_4\cdots z_{m-2}b_mx.
	\]
	If $n=2m+1$, then $|Z|=m-1$ and a Hamilton cycle is
	\[
	xyb_2z_1b_3z_2\cdots b_mz_{m-1}b_1x.
	\]
	Consequently, $G$ is $1$-tough.
	
	It remains to check the degree sequence condition. The degrees are
	\[
	\begin{aligned}
		d_G(x)&=\lfloor m/2\rfloor+1,
		& d_G(y)&=\lceil m/2\rceil+1,\\
		d_G(z)&=m\quad(z\in Z),
		& d_G(b)&=\lceil n/2\rceil-1\quad(b\in B).
	\end{aligned}
	\]
	Let $d_1,d_2,\ldots,d_n$ be the degree sequence of $G$, and put
	$h=\lceil n/2\rceil-1$. If $n$ is even, the first two terms are
	$\lfloor m/2\rfloor+1$ and $\lceil m/2\rceil+1$, followed by
	$m$ terms equal to $m-1$ and $m-2$ terms equal to $m$.
	If $n$ is odd, the first two terms are the same and all remaining
	$2m-1$ terms are equal to $m$. Since $m\ge 5$, in both cases
	$d_i>i$ for every $i\in[1,h-1]$, while
	$d_h=d_{n-h+1}=h<n-h$.
	Thus $h$ is the only index for which $d_h\le h<n/2$ and
	$d_{n-h+1}<n-h$. There is no integer $j$ with $h<j<n/2$, so
	the required degree-sum condition holds vacuously at this index.
	At the omitted midpoint index $j=\lceil n/2\rceil$, the
	degree sum $d_j+d_{n-j+1}$ is $n-2$ when $n$ is even and
	$n-1$ when $n$ is odd. Thus these graphs fail precisely the
	additional degree-sum requirement in
	Theorem~\ref{thm_hoang-tough}.
	This completes the construction.
\end{proof}

\section{Proof of Theorem~\ref{thm_main-theorem}}
\label{sec_main-proof}

\begin{proof}
	Let $t\ge 7$ be an integer and $G$ be a graph satisfying the
	hypotheses of Theorem~\ref{thm_main-theorem}.
By Theorem~\ref{thm_shan-tanyel-hamiltonicity}, $G$ is Hamiltonian. We may assume that $G$ is not complete.
	By Lemma~\ref{lem_toughness-observation},
	$\delta(G)\ge\kappa(G)\ge 2t$ and $\alpha(G)\le n/(t+1)$.
	Label the vertices as $v_1,\ldots,v_n$ so that $d_G(v_i)=d_i$,
	and put $p=\lceil n/2\rceil-1$.
	
	If there is no integer $h\in[t,p]$ with $d_h\le h$ and
	$d_{n-h+t}<n-h$, then Lemma~\ref{lem_hoang-no-bad-index} implies
	that $G$ is pancyclic or bipartite. Since a bipartite graph on
	at least three vertices has toughness at most $1$, $G$ is
	pancyclic. We therefore assume that such an index $h$ exists.
	In particular, $2t\le\delta(G)\le d_h\le h<n/2$.
	
	\begin{claim}
		\label{claim_uniqueness-of-h}
		The index $h$ is unique.
	\end{claim}
	
	\begin{proof}[Proof of Claim~\ref{claim_uniqueness-of-h}]
		Suppose that $h_1<h_2$ are two such indices. Applying the degree
		sequence condition at $h_1$ gives
		$d_{h_2}+d_{n-h_2+t}\ge n$, contrary to
		$d_{h_2}\le h_2$ and $d_{n-h_2+t}<n-h_2$.
	\end{proof}
	
	Let $R=\{v\in V(G):d_G(v)>n/2\}$. We record a consequence that
	will avoid repeating the same argument in two cases below:
	\begin{equation}
		\label{eq_small-order}
 \text{If} \quad	n\le 12t\ \text{ and }\ |R|\ge 2, 
		\quad  \text{then} \quad  G \quad \text{is pancyclic}.
	\end{equation}
	Indeed, since $t\ge 7$,
	\[
	\alpha(G)\le\frac{n}{t+1}\le\frac{12t}{t+1}
	<2t-1\le\kappa(G)-1.
	\]
	Thus $G$ is Hamiltonian-connected by
	Lemma~\ref{lem_chvatal-hamiltonian-connected}. Choose distinct
	$x,y\in R$ and a Hamilton path with endpoints $x$ and $y$.
	Since $d_G(x)+d_G(y)>n$, Lemma~\ref{lem_bondy-hamilton-cycle}
	applies if $xy\in E(G)$, and
	Lemma~\ref{lem_faudree-hamilton-path} applies otherwise.
	This proves~\eqref{eq_small-order}.
	
	\medskip
	\noindent\textbf{Case 1.} $h<p$.
	
	If $d_p>p$, then at least $n-p+1\ge n/2$ vertices have degree
	at least $n/2$. By Lemma~\ref{lem_hoang-R-half}, $G$ is
	pancyclic, bipartite, or isomorphic to $S_n$. The last two
	possibilities are excluded by $t\ge 7$ and
	$\delta(G)\ge 2t>2$.
	
	Suppose that $d_p\le p$. By Claim~\ref{claim_uniqueness-of-h},
	$d_{n-p+t}\ge n-p>n/2$, so $|R|\ge p-t+1$.
	Since $p\ge d_p\ge 2t$, we have $|R|\ge t+1\ge 2$.
	If $n\le 12t$, apply~\eqref{eq_small-order}.
	Otherwise, $|R|\ge p-t+1\ge n/2-t>n/3$, and
	Lemma~\ref{lem_hoang-R-third} implies that $G$ is pancyclic.
	
	\medskip
	\noindent\textbf{Case 2.} $h=p$.
	
	We distinguish three subcases according to the indices below $h$
	for which $d_i\le i$.
	
	\medskip
	\noindent\textbf{Case 2.1.} $d_i>i$ for every $i\in[1,h-1]$.
	
	Let $G^*=C_{n-t+5}(G)$. Since $h-1<d_{h-1}\le d_h\le h$,
	we have $d_{h-1}=d_h=h$. For distinct $i,j\in[h-1,n]$,
	$d_i+d_j\ge 2h\ge n-2\ge n-t+5$. Hence
	$\{v_{h-1},v_h,\ldots,v_n\}$ is  a clique in $G^*$.
	
	Suppose that $\{v_j,\ldots,v_n\}$ is  a clique in $G^*$
	for some $j\in[2,h-1]$. For every $i\in[j,n]$,
	\[
	d_{G^*}(v_{j-1})+d_{G^*}(v_i)
	\ge d_{j-1}+n-j\ge n>n-t+5.
	\]
	Thus $v_{j-1}$ is adjacent in $G^*$ to every vertex of this
	clique. Repeating this argument shows that $G^*$ is complete.
	By Lemma~\ref{lem_s-hamiltonian-closure} with $s=3$, $G$ is
	$3$-Hamiltonian. Since $\Delta(G)\ge h>(n-3)/2$,
	Lemma~\ref{lem_s-hamiltonian} implies that $G$ is pancyclic.
	
	\medskip
	\noindent\textbf{Case 2.2.} There is an index $i<h$ with
	$d_i\le i$, and the largest such index $k$ satisfies
	$k\ge n/2-t$.
	
	By Claim~\ref{claim_uniqueness-of-h},
	$d_{n-k+t}\ge n-k>n/2$, so $|R|\ge k-t+1$.
	Also, $k\ge d_k\ge 2t$, and hence $|R|\ge t+1\ge 2$.
	If $n\le 12t$, apply~\eqref{eq_small-order}.
	If $n>12t$, then
	\[
	|R|\ge k-t+1\ge\frac n2-2t+1>\frac n3,
	\]
	and the conclusion follows from Lemma~\ref{lem_hoang-R-third}.
	
	\medskip
	\noindent\textbf{Case 2.3.} There is an index $i<h$ with
	$d_i\le i$, and the largest such index $k$ satisfies
	$k<n/2-t$.
	
	In this case $k<h-1$, so $d_{h-1}=d_h=h$.
	Also, Claim~\ref{claim_uniqueness-of-h} gives
	\begin{equation}
		\label{eq_large-maximum-degree}
		\Delta(G)\ge d_{n-k+t}\ge n-k>\frac n2.
	\end{equation}
	Let $G^*=C_{n-t+2}(G)$. It is enough to show that $G^*$ is
	$1$-Hamiltonian: Lemma~\ref{lem_s-hamiltonian-closure} then gives
	that $G$ is $1$-Hamiltonian, and
	Lemma~\ref{lem_s-hamiltonian} applies by
	\eqref{eq_large-maximum-degree}.
	
	We will find a clique $Q$ such that each component of $G^*-Q$
	is an isolated vertex or an edge. We will then choose three
	attachment vertices in $Q$ for each component, with these choices
	pairwise disjoint, and use them to construct Hamilton cycles
	after deleting at most one vertex.
	
	Order the vertices of $G^*$ as $u_1,\ldots,u_n$ so that
	$d_i^*=d_{G^*}(u_i)$ and $d_1^*\le\cdots\le d_n^*$.
	The orderings of $V(G)$ and $V(G^*)$ need not agree. Nevertheless,
	$d_i\le d_i^*$ for every $i$: otherwise at least $i$ vertices
	would have degree less than $d_i$ in $G^*$, and hence also in
	$G$, contrary to the definition of $d_i$.
	By the definition of closure,
	\begin{equation}
		\label{eq_nonedge-closure}
		u_iu_j\notin E(G^*)
		\quad \text{implies}\quad d_i^*+d_j^*\le n-t+1.
	\end{equation}
	
	Since $d_{h-1}^*\ge h$ and $2h\ge n-2\ge n-t+2$, the set
	$\{u_{h-1},u_h,\ldots,u_n\}$ forms  a clique in $G^*$.
	If $G^*$ is complete, it is $1$-Hamiltonian and we are done.
	Otherwise, choose $q$ as small as possible so that
	$Q=\{u_{q+1},\ldots,u_n\}$ forms a clique in $G^*$.
	Then $1\le q\le h-2$ and $|Q|>n/2$.
	
	By the choice of $q$, the vertex $u_q$ has a nonneighbor $u_j$
	in $Q$. Since $d_j^*\ge n-q-1$, \eqref{eq_nonedge-closure}
	gives $d_q^*\le q-t+2<q$. Thus $d_q\le q$, and the choice
	of $k$ implies $q\le k$. Moreover, for every $i\in[1,q]$,
	$d_i^*\le d_q^*<q<|Q|$, so $u_i$ has a nonneighbor in $Q$.
	Let $\ell_i$ be the largest index of such a nonneighbor, and put
	$T_i=\{u_{\ell_i+1},\ldots,u_n\}$. Then
	$T_i\subseteq N_{G^*}(u_i)$ and $n-\ell_i\le d_i^*$.
	
	\begin{claim}
		\label{claim_clique-neighborhoods}
		For every $i\in[1,q]$, $|N_{G^*}(u_i)\setminus T_i|\le 1$.
	\end{claim}
	
	\begin{proof}[Proof of Claim~\ref{claim_clique-neighborhoods}]
		We show that $d_i^*\le n-\ell_i+1$ by downward induction on $i$.
		Suppose the assertion holds for all indices in $[i+1,q]$.
		If there is an index $j\in[i+1,q]$ with $\ell_j\ge\ell_i$, then
		$d_i^*\le d_j^*\le n-\ell_j+1\le n-\ell_i+1$, as required.
		
		We may therefore assume that $\ell_i>\ell_j$ for every
		$j\in[i+1,q]$. The vertex $u_{\ell_i}$ is then adjacent to
		$u_{i+1},\ldots,u_q$ and to every vertex of $Q$ other than itself.
		Consequently, $d_{\ell_i}^*\ge n-i-1$, and
		\eqref{eq_nonedge-closure} gives $d_i^*\le i-t+2$.
		Put $r=t+d_i^*-2$. Since $d_i^*\ge\delta(G)\ge 2t$,
		we have $t\le r\le i<h$.
		
		Suppose to the contrary that $\ell_i\ge n-d_i^*+2$.
		Using $r\le i$, the domination $d_j\le d_j^*$, and
		\eqref{eq_nonedge-closure}, we obtain
		\[
		\begin{aligned}
			d_r&\le d_i\le d_i^*\le r,\\
			d_{n-r+t}=d_{n-d_i^*+2}
			&\le d_{\ell_i}\le d_{\ell_i}^*
			\le n-t+1-d_i^*=n-r-1.
		\end{aligned}
		\]
		Thus $r<h$ is another index with $d_r\le r$ and
		$d_{n-r+t}<n-r$, contrary to
		Claim~\ref{claim_uniqueness-of-h}. Therefore
		$\ell_i\le n-d_i^*+1$, or equivalently,
		$d_i^*\le n-\ell_i+1$. When $i=q$, the set $[i+1,q]$ is empty, so there is no
		index $j$ to which the induction hypothesis could apply.
		The argument starting with the degree bound on $u_{\ell_i}$
		therefore applies directly and establishes the base case.
	\end{proof}
	
	For $x=u_i\notin Q$, write $T(x)=T_i$. Each set $T(x)$ consists of all vertices after a specified
	position in the fixed ordering $u_{q+1},u_{q+2},\ldots,u_n$ of $Q$.
	Consequently, if $\ell_i\le\ell_j$, then $T_j\subseteq T_i$,
	so any two of these sets are nested.
	Claim~\ref{claim_clique-neighborhoods} also implies that each
	vertex of $G^*-Q$ has at most one neighbor in $G^*-Q$.
	Hence every component of $G^*-Q$ is an isolated vertex or an edge.
	
	Let these components be $D_1,\ldots,D_r$.
	If $D_i$ consists of a single vertex, denote it by $x_i$ and call
	it the primary vertex. If $D_i$ is an edge, write $D_i=x_iy_i$ so that
	$T(x_i)\subseteq T(y_i)$, and call $x_i$ its primary vertex
	and $y_i$ its secondary vertex. In the latter case, the edge
	$x_iy_i$ accounts for the possible neighbor outside each terminal
	segment. Thus
	\begin{equation}
		\label{eq_common-neighbors}
		N_{G^*}(x_i)\cap Q=T(x_i)\subseteq T(y_i)
		=N_{G^*}(y_i)\cap Q.
	\end{equation}
	Let $X=\{x_1,\ldots,x_r\}$, and let $B$ be the bipartite graph
	with bipartition $(X,Q)$ consisting of all edges of $G^*$ between
	these two sets.
	
	\begin{claim}
		\label{claim_matching}
		There are pairwise disjoint sets $A_1,\ldots,A_r\subseteq Q$
		such that $|A_i|=3$ and every vertex of $A_i$ is adjacent to
		every vertex of $D_i$ for each $i\in[1,r]$.
	\end{claim}
	
	\begin{proof}[Proof of Claim~\ref{claim_matching}]
		By Lemma~\ref{lem_generalized-marriage} with $f(x)=3$ for all
		$x\in X$, it suffices to show that $|N_B(S)|\ge 3|S|$ for
		every $S\subseteq X$. The assertion is immediate for $S=\emptyset$.
		If $S=\{x_i\}$, then
		$|N_B(S)|\ge d_{G^*}(x_i)-1\ge 2t-1\ge 3$.
		
		Now suppose that $|S|\ge 2$ and $|N_B(S)|<3|S|$.
		Put $W=N_{G^*}(S)$. Since $X$ contains one vertex from each
		component of $G^*-Q$, it is independent. Each vertex of $S$
		has at most one neighbor outside $Q$, so
		$|W|\le |N_B(S)|+|S|<4|S|$.
		Every vertex of $S$ is isolated in $G^*-W$, and hence
		$c(G^*-W)\ge |S|\ge 2$. This contradicts the $t$-toughness
		of $G^*$, since
		$|W|<4|S|\le t\,c(G^*-W)$.
		
		The generalized marriage lemma now gives pairwise disjoint
		three-element sets $A_i\subseteq N_B(x_i)$.
		For an edge component, \eqref{eq_common-neighbors} shows that
		$A_i$ is also contained in the neighborhood of its secondary
		vertex. This proves the claim.
	\end{proof}
	
	\begin{claim}
		\label{claim_1-hamiltonian}
		The graph $G^*$ is $1$-Hamiltonian.
	\end{claim}
	
	\begin{proof}[Proof of Claim~\ref{claim_1-hamiltonian}]
		Let $S\subseteq V(G^*)$ with $|S|\le 1$, and put $Q'=Q\setminus S$.
		For every $i$ with $V(D_i)\setminus S\ne\emptyset$, choose distinct
		$a_i,b_i\in A_i\setminus S$. These choices are possible because
		$|A_i|=3$. By Claim~\ref{claim_matching}, the edges $a_ib_i$
		form a matching in the complete graph $G^*[Q']$.
		Also, $|Q'|\ge 3$, since $|Q|>n/2$ and $n>4t\ge 28$.
		Thus $G^*[Q']$ has a Hamilton cycle containing all these edges.
		
		For each surviving component $D_i-S$, replace $a_ib_i$ on this
		cycle by a path from $a_i$ to $b_i$ through all vertices of
		$D_i-S$. Such a path exists because $D_i-S$ is a single vertex
		or an edge and both $a_i$ and $b_i$ are adjacent to every one
		of its vertices. The replacements are disjoint and give a
		Hamilton cycle of $G^*-S$. Since $S$ was arbitrary, $G^*$ is
		$1$-Hamiltonian.
	\end{proof}
	
	Lemma~\ref{lem_s-hamiltonian-closure} now implies that $G$ is
	$1$-Hamiltonian. By \eqref{eq_large-maximum-degree} and
	Lemma~\ref{lem_s-hamiltonian}, $G$ is pancyclic. This completes
	Case 2.3 and the proof of the theorem.
\end{proof}

\section*{Declaration on Use of AI}
ChatGPT Pro was used to polish the language of an initial draft prepared by the authors. The authors reviewed and verified all AI-assisted revisions and take full responsibility for the content of the manuscript.

\bibliographystyle{abbrv}
\bibliography{pancyclic}

\end{document}